\documentclass[12pt,a4paper]{article}

\usepackage{mathrsfs}
\usepackage{amsmath,amsthm,color,eepic}
\usepackage{amssymb} \usepackage{verbatim}

\newtheorem{theorem}{Theorem}[section]

\newtheorem{lemma}[theorem]{Lemma}

\theoremstyle{definition}

\newtheorem{claim}{Claim}

\newtheorem{case}{Case}
\newtheorem{subcase}{Case}[case]

\newcommand{\ar}{{\rm ar}}
\newcommand{\ex}{{\rm ex}}

\begin{document}
\title {The anti-Ramsey numbers of cliques in complete multi-partite graphs}

\author{Yuyu An$^{a,b}$,~Ervin Gy\H{o}ri$^{c,}$\thanks{Corresponding author. E-mail: anyuyu@mail.nwpu.edu.cn (Y.A.), gyori.ervin@renyi.hu (E.G.), binlongli@nwpu.edu.cn (B.L.)}~, Binlong Li$^{a,b}$\\[2mm]
{\small $^a$School of Mathematics and Statistics,}\\[-0.8ex]
{\small Northwestern Polytechnical University, Xi'an, 710072, China}\\
{\small $^b$Xi'an-Budapest Joint Research Center for Combinatorics,}\\[-0.8ex]
{\small Northwestern Polytechnical University, Xi'an, 710072, China}\\
{\small $^c$Alfr\'{e}d R\'{e}nyi Institute of Mathematics, }\\[-0.8ex]
{\small Hungarian Academy of Sciences, 1053 Budapest, Hungary}
}

\date{}

\maketitle

\begin{abstract}
  A subgraph of an edge-colored graph is rainbow if all of its edges have different colors. Let $G$ and $H$ be two graphs. The anti-Ramsey number $\ar(G, H)$ is the maximum number of colors of an edge-coloring of $G$ that does not contain a rainbow copy of $H$. In this paper, we study the anti-Ramsey numbers of $K_k$ in complete multi-partite graphs. We determine the values of the anti-Ramsey numbers of $K_k$ in complete $k$-partite graphs and in balanced complete $r$-partite graphs for $r\geq k$.
\smallskip

\emph{Keywords:} anti-Ramsey number; multi-partite graph; extremal coloring
\end{abstract}

\section{Introduction}

Let $G$ be a graph, we use $e(G)$ to denote the number of edges of $G$. An \emph{edge-coloring} of $G$ is a mapping $c: E(G)\rightarrow\mathbb{N}$, where $\mathbb{N}$ is the set of natural numbers. We call $G$ an \emph{edge-colored} graph if it is assigned such an edge-coloring $c$. A subgraph $H$ of $G$ is called \emph{rainbow} if all of its edges have different colors.

For given graphs $G$ and $H$, the \emph{Tur\'an number} $\ex(G,H)$ is the maximum number of edges in a subgraph of $G$ without copy of $H$; and the subgraphs achieving the maximum edge number are \emph{extremal} for $\ex(G,H)$. The \emph{anti-Ramsey number} $\ar(G, H)$ is the maximum number of colors in an edge-coloring of $G$ without rainbow copy of $H$; and the edge-colorings achieving the maximum color number are \emph{extremal} for $\ar(G, H)$. Clearly $\ar(G,H)\leq\ex(G,H)$.

The study of anti-Ramsey theory was initiated by Erd\H{o}s, Simonovits and S\'{o}s \cite{ErSiSo} and considered in the classical case when $G=K_n$. Since then plentiful results were established for a variety of graphs $H$, including among cliques, cycles, paths, etc. We refer the read to \cite{FuMaOz} for a survey.

Erd\H{o}s, Simonovits and S\'{o}s \cite{ErSiSo} calculated  the anti-Ramsey numbers $\ar(K_n, K_3)$. They also determined $\ar(K_n, K_k)$ for $n$ large enough and $k\geq 4$ in the same paper. Schiermeyer \cite{Sc} showed that the result of $\ar(K_n, K_k)$ given by Erd\H{o}s et al. holds for all $n\geq k\geq 4$.

\begin{theorem}[Erd\H{o}s et al. \cite{ErSiSo}]\label{ThErSiSo}
For all $n\geq 3$, $\ar(K_n,K_3)=n-1$.
\end{theorem}

\begin{theorem}[Schiermeyer \cite{Sc}]\label{ThSc}
For all $n\geq k\geq 4$, $\ar(K_n,K_k)=\ex(K_n,K_{k-1})+1$.
\end{theorem}

Fang et al. studied the anti-Ramsey numbers of $K_3$ in complete multi-partite graphs.

\begin{theorem}[Fang et al. \cite{FaGyLiXi}]\label{ThFaGyLiXi}
For $r\geq 3$ and $n_1\geq n_2\geq\cdots\geq n_r\geq 1$, we have
$$\ar(K_{n_1,n_2,\ldots,n_r},K_3)=\left\{\begin{array}{ll}
  n_1n_2+n_3n_4+\cdots+n_{r-2}n_{r-1}+n_r+\frac{r-1}{2}-1, & r \mbox{ is odd};\\
  n_1n_2+n_3n_4+\cdots+n_{r-1}n_r+\frac{r}{2}-1, & r \mbox{ is even}.
\end{array}\right.$$
\end{theorem}

In this paper, we consider the anti-Ramsey number of $K_k$ in complete multi-partite graphs. We first give the anti-Ramsey numbers of $K_k$ in complete $k$-partite graphs and in balanced complete $r$-partite graphs with $r\geq k$.

\begin{theorem}\label{ThkPartite}
For $k\geq 3$ and $n_1\geq n_2\geq\cdots\geq n_k\geq 1$, we have
$$\ar(K_{n_1,n_2,\ldots,n_k},K_k)=\sum_{1\leq i<j\leq k}n_in_j-n_k(n_{k-1}+n_{k-2}-1).$$
\end{theorem}

Let $K_r^t$ be the complete $r$-partite graph with all partite sets of equal size $t$.

\begin{theorem}\label{ThBalanced}
For $r\geq k\geq 4$, we have $$\ar(K_r^t,K_k)=\left\{\begin{array}{ll}
  t^2({k\choose 2}-2)+t, & r=k\\
  t^2\ex(K_r,K_{k-1})+1, & r>k.
\end{array}\right.$$
\end{theorem}

For the anti-Ramsey numbers of $K_k$ in unbalanced complete $r$-partite graphs, we can only give some structure properties of extremal colorings, see Theorem \ref{ThUnbalanced} in next section.

This paper is organized as follows: In Section 2, we will introduce basic terminology and significant lemmas. In Section 3, we give the proof of Theorems \ref{ThkPartite} and \ref{ThBalanced}. In Section 4, we summarize the extremal coloring for the anti-Ramsey numbers of $K_k$ in complete multi-partite graphs.

\section{Some preliminaries}

\subsection{Symmetrization of graphs and colorings}

Throughout the paper, a coloring always infer to an edge-coloring. Let $G$ be a colored graph with coloring $c$ and $v\in V(G)$. We set $C(G)=\{c(e): e\in E(G)\}$ and $c(G)=|C(G)|$. We say a color $a\in C(G)$ \emph{appears} at $v$ if there is at least one edge of color $a$ incident to $v$, and we say $a$ is \emph{saturated} by $v$ if every edge of color $a$ is incident to $v$ (and $a$ appears at $v$). Note that if a color $a$ is saturated by a vertex $v$, then $G^a$, the subgraph of $G$ induced by all the edges of color $a$, is a star centered at $v$. We use $C_G(v)$ and $S_G(v)$ to denote the set of colors that appears at $v$ and that is saturated by $v$, respectively. We define the color degree and the saturated color degree of $v$ as $d_G^c(v)=|C_G(v)|$ and $d_G^s(v)=|S_G(v)|$, we write $d^c(v)$ and $d^s(v)$ instead of $d_G^c(v)$ and $d_G^s(v)$ for short. Note that every color is saturated by 0, 1 or 2 vertices, and a color $a$ is saturated by 2 vertices if and only if $G$ has exactly one edge of color $a$. In this case we call $a$ an \emph{exclusive} color. We use $S^i(G)$, $i=0,1,2$, to denote the set of colors that saturated by $i$ vertices of $G$.

For a set $U\subset V(G)$, we say a color $a$ is saturated by $U$ if every edge of color $a$ is incident to some vertex in $U$.

To stress the coloring, we sometime denote by $G^c$ the colored graph $G$ with coloring $c$.

Let $G$ be a (non-colored) graph and $u,v\in V(G)$ be nonadjacent. We say $u$ and $v$ are symmetric in $G$ if $N_G(u)=N_G(v)$. The symmetrization of $G$ at $v$ to $u$, is the graph $G'$ obtained from $G$ by removing all edges incident to $v$ and then adding all edges in $\{vx: ux\in E(G)\}$. Note that if $G'$ is a symmetrization of $G$ at $v$ to $u$, then $e(G')=e(G)-d(v)+d(u)$.

Let $G^c$ be a colored graph and $u,v\in V(G)$ be nonadjacent. We say $u$ and $v$ are symmetric in $G^c$ if they are symmetric in $G$ (i.e., $N_G(u)=N_G(v)$) and there is a bijection $\sigma: S_G(u)\rightarrow S_G(v)$, such that for every vertex $x\in N_G(v)$, $c(vx)=c(ux)$ if $c(ux)\notin S_G(u)$ and $c(vx)=\sigma c(ux)$ if $c(ux)\in S_G(u)$. Suppose that $u,v$ are symmetric in $G$. The symmetrization of the coloring $c$ at $v$ to $u$, is the coloring $c'$ of $G$ obtained from $c$ by the following operation: first for each color $a\in S_G(u)$, define an extra color $\sigma a$ ($\notin C(G)$), and then recolor the edges incident to $v$ such that $c'(vx)=c(ux)$ if $c(ux)\notin S_G(u)$ and $c'(vx)=\sigma c(ux)$ if $c(ux)\in S_G(u)$. Notice that if $c'$ is the symmetrization of $c$ at $v$ to $u$, then $u,v$ are symmetric in $G^{c'}$.

\begin{lemma}\label{LeSymmetrization}
Let $G^c$ be a colored graph, $u,v$ be symmetric in $G$, and let $c'$ be the symmetrization of $c$ at $v$ to $u$. Then\\
(1) $c'(G)=c(G)-d^s(v)+d^s(u)$;\\
(2) if $G^{c'}$ contains a rainbow $K_k$, then so does $G^c$;\\
(3) if two vertices $x,y\in V(G)\backslash\{v\}$ are symmetric in $c$, then they are symmetric in $c'$.
\end{lemma}

\begin{proof}
  (1) Let $\sigma a$ be the extra color corresponding to $a\in S_G(u)$ by the definition of the symmetrization. Then $C'(G)=(C(G)\backslash S_G(v))\cup\{\sigma a: a\in S_G(u)\}$. Since $\sigma a\notin C(G)$ for every $a\in S_G(u)$, we have that $c'(G)=c(G)-d^s(v)+d^s(u)$.

  (2) Suppose that $H$ is a rainbow $K_k$ in $G^{c'}$. Since $uv\notin E(G)$, either $u$ or $v$ is not contained in $H$. Notice that the only recolored edges in $c'$ are those incident to $v$. If $v\notin V(H)$, then $H$ is also rainbow in $G^c$. Now we assume that $v\in V(H)$ and $u\notin V(H)$. Let $H'=G[(V(H)\backslash\{v\})\cup\{u\}]$. For every vertex $x\in V(H)\backslash\{v\}$, $c'(vx)=c'(ux)$ if $c(ux)\notin S_G(u)$ and $c'(vx)=\sigma c(ux)$ if $c(ux)\in S_G(u)$. This implies that $H'$ is rainbow in $G^{c'}$, and then in $G^c$, as well.

  (3) Let $\eta: S_G(x)\rightarrow S_G(y)$ be the bijection such that for every $z\in N(y)$, $c(yz)=c(xz)$ if $c(xz)\notin S_G(x)$, and $c(yz)=\eta c(xz)$ if $c(xz)\in S_G(x)$. We use $S'_G(x), S'_G(y)$ to denote the colors saturated by $x, y$, respectively, in $G^{c'}$. If $ux\notin E(G)$ (including the case $u=x$ or $u=y$), then $uy,vx,vy\notin E(G)$ since $x,y$ are symmetric and $u,v$ are symmetric in $G$. Note that the symmetrization of $c$ at $v$ to $u$ only change the colors of edges incident to $v$. We see that $S'_G(x)=S_G(x)$, $S'_G(y)=S_G(y)$ and for every $z\in N(x)$, $c'(xz)=c(xz), c'(yz)=c(yz)$. It follows that $x,y$ are symmetric in $G^{c'}$.

  Now we suppose that $xu\in E(G)$, and then $xv,yu,yv\in E(G)$. Let $a=c(xu)$. If $a\notin S_G(x)$ and $a\notin S_G(u)$, then $S'_G(x)=S_G(x)$, $S'_G(y)=S_G(y)$ and $c'(xu)=c'(yu)=c'(xv)=c'(yv)=a$. If $a\notin S_G(x)$ and $a\in S_G(u)$, then  $S'_G(x)=S_G(x)$, $S'_G(y)=S_G(y)$ and $c'(xu)=c'(yu)=a$, $c'(xv)=c'(yv)=\sigma a$. If $a\in S_G(x)$ and $a\notin S_G(u)$, then $S'_G(x)=S_G(x)$, $S'_G(y)=S_G(y)$ and $c'(xu)=c'(xv)=a$, $c'(yu)=c'(yv)=\eta a$. For each case it follows that $x,y$ are symmetric in $G^{c'}$.

  Finally suppose that $a\in S_G(x)\cap S_G(u)$. We have that $c(yu)=\eta a\in S_G(y)\cap S_G(u)$. It follows that $c'(xu)=a, c'(yu)=\eta a, c'(xv)=\sigma a, c'(yv)=\sigma\eta a$ and $S'_G(x)=(S_G(x)\backslash\{c(xv)\})\cup\{\sigma a\}$, $S'_G(y)=(S_G(y)\backslash\{c(yv)\})\cup\{\sigma\eta a\}$. We define a bijection $\eta': S'_G(x)\rightarrow S'_G(y)$ such that $\eta'\sigma a=\sigma\eta a$, and $\eta'b=\eta b$ for all $b\in S'_G(x)\backslash\{\sigma a\}$. Then $x,y$ are symmetric in $G^{c'}$ (with the corresponding bijection $\eta'$).
\end{proof}

\begin{lemma}\label{LeExtremal}
Let $G^c$ be a colored complete multi-partite graph without rainbow $K_k$ such that $c(G)$ is as large as possible. Then\\
(1) if a color $a$ is saturated by a partite set $U$, then $a$ is saturated by a vertex in $U$;\\
(2) if two vertices $x,y$ are in a common partite set, then $d^s(x)=d^s(y)$.
\end{lemma}

\begin{proof}
  (1) Suppose that the color $a$ is not saturated by any vertices in $U$. Let $u_i$, $1\leq i\leq s$, be the vertices in $U$ with $a\in C_G(u_i)$, where $s\geq 2$. We define a coloring $c'$ of $G$ such that $c'(e)=a_i$ if $e$ is incident to $u_i$ and $c(e)=a$, and $c'(e)=c(e)$ otherwise, where $a_1,\ldots,a_s$ are $s$ extra colors. Thus $c'(G)=c(G)-1+s>c(G)$, implying that $G^{c'}$ contains a rainbow $K_k$. Let $H$ be a rainbow $K_k$ in $G^{c'}$. Since $u_1,\ldots,u_s\in U$ which are independent, $H$ contains at most one vertex in $\{u_1,\ldots,u_s\}$. It follows that $H$ is also rainbow in $G^c$, a contradiction.

  (2) Suppose that $d^s(x)>d^s(y)$. Let $c'$ be the symmetrization of $c$ at $y$ to $x$. By Lemma \ref{LeSymmetrization}, $c'(G)>c(G)$ and $G^{c'}$ contains no rainbow $K_k$, contradicting the choice of $c$.
\end{proof}

We say a coloring $c$ of a complete multi-partite graph $G$ is \emph{totaly symmetric} if each two vertices that contained in a common partite set are symmetric in $G^c$.

\begin{lemma}\label{LeTotaly}
Let $G$ be a complete multi-partite graph and $H$ a complete graph. Then there is a totaly symmetric coloring $c$ of $G$ with $c(G)=\ar(G,H)$.
\end{lemma}

\begin{proof}
  Suppose $c_0$ is an extremal coloring for $\ar(G,H)$, i.e., $G^{c_0}$ contains no rainbow $H$ and $c_0(G)=\ar(G,H)$. Let $U_i$, $1\leq i\leq r$, be the partite sets of $G$. We define a series of colorings $c_1,c_2,\ldots,c_t$ of $G$ such that\\
  (i) for each $j$, $0\leq j<t$, there are two distinct vertices $u_i,u'_i\in U_i$ for some $i$, $1\leq i\leq r$, such that $u_i,u'_i$ are not symmetric in $G^{c_j}$ and $c_{j+1}$ is the symmetrization of $c_j$ at $u'_i$ to $u_i$; and \\
  (ii) for each two vertices $u_i,u'_i\in U_i$, $1\leq i\leq r$, $u_i,u'_i$ are symmetric in $G^{c_t}$.\\
  By Lemma \ref{LeSymmetrization} (3), the operation above is terminable since $|V(G)|$ is finite. Notice that $c_t$ is a totaly symmetric coloring of $G$.

  By Lemma \ref{LeSymmetrization} (2), each $G^{c_k}$ contains no rainbow $H$. By Lemma \ref{LeExtremal}, each two vertices of $G$ in a common partite set have the same saturated color degree for all colorings $c_j$, $1\leq j\leq t$. By Lemma \ref{LeSymmetrization} (1), $c_0(G)=c_1(G)=\cdots=c_t(G)=\ar(G,H)$. Therefore, $c_t$ is a totaly symmetric coloring of $G$ with $c_t(G)=\ar(G,H)$.
\end{proof}

\subsection{Blow-up of graphs and colorings}

Let $G$ be a graph on $V(G)=\{v_1,\ldots,v_n\}$, and $f: V(G)\rightarrow\mathbb{N}$ be a size function. The \emph{blow-up} of $G$ with $f$ is the graph $\mathcal{B}(G,f)$ on vertex set $\bigcup\{U_i: v_i\in V(G)\}$, where $U_i=\{u_i^1,\ldots,u_i^{f(v_i)}\}$ is a set corresponding to $v_i$ ($U_i\cap U_j=\emptyset$ if $v_i\neq v_j$), such that for each $u_i^s, u_j^t$, $u_i^su_j^t\in E(\mathcal{B}(G,f))$ if and only if $v_iv_j\in E(G)$. Note that $K_{n_1,\ldots,n_r}$ is a blow-up of $K_r$ with the size function $f(v_i)=n_i$, $1\leq i\leq r$.

Let $G^c$ be a colored graph and $\mathcal{B}(G,f)$ be a blow-up of $G$. We define the \emph{blow-up} of $G^c$, denoted by $\mathcal{B}(G^c,f)$, as the colored graph $\mathcal{B}(G,f)$ with a coloring (also denoted as $c$) such that\\
(1) if $a\in S^0(G)$, then arrange $c(u_i^su_j^t)=a$ if $c(v_iv_j)=a$; \\
(2) if $b\in S^1(G)$ is saturated by a vertex $v_i$, then we define a set of $f(v_i)$ colors $\{b_s: 1\leq s\leq f(v_i)\}$, and arrange $c(u_i^su_j^t)=b_s$ if $c(v_iv_j)=b$; \\
(3) if $c\in S^2(G)$ is saturated by two vertices $v_i,v_j$, then we define a set of $f(v_i)f(v_j)$ colors $\{c_{s,t}: 1\leq s\leq f(v_i), 1\leq t\leq f(v_j)\}$, and arrange $c(u_i^su_j^t)=c_{s,t}$.

In the following of the paper, we always set $H=K_k$, $K=K_r$ with vertex set $V(K)=\{v_1,v_2,\ldots,v_r\}$, and $G=K_{n_1,\ldots,n_r}$ with partite sets $U_i$, $1\leq i\leq r$, where $U_i=\{u_i^1,u_i^2,\ldots,u_i^{n_i}\}$. We let $f$ be the size function on $V(K)$ with $f(v_i)=n_i$, $1\leq i\leq r$.

\begin{theorem}\label{ThUnbalanced}
For all $r\geq k\geq 3$ and $n_1\geq n_2\geq\cdots\geq n_r$, we have
$$\ar(K_{n_1,\ldots,n_r},K_k)=\max\{c(\mathcal{B}(K_r^c,f)): c\mbox{ is a coloring of }K_r\mbox{ without rainbow }K_k\},$$
where $f$ is a size function on $V(K_r)$ with $f(v_i)=n_i$, $1\leq i\leq r$.
\end{theorem}

\begin{proof}
  By Lemma \ref{LeTotaly}, there is an extremal coloring $c$ for $\ar(G,H)$ that is totaly symmetric. We define a coloring of $K$, also denoted by $c$, such that $c(v_iv_j)=c(u_i^1u_j^1)$ (i.e., $K^c$ is isomorphic to the colored subgraph of $G^c$ induced by $\{u_i^1,\ldots,u_r^1\}$). So $K^c$ is a colored $K_r$ without rainbow $H=K_k$. Recall that $G$ is a blow-up of $K$ with the size function $f$. We will show that $G^c$ is the blow-up of $K^c$. Notices that $S^i(K)$, $i=0,1,2$, is the set of colors that saturated by $i$ vertices in $K^c$.

  Assume first that $a$ is a color in $S^0(K)$. For any edge $v_iv_j$ of $K$ with color $a$, $a$ is not saturated by $v_i,v_j$ in $K^c$. This implies that $a$ is not saturated by $u_i^1,v_i^1$ in $G^c$. Since $G^c$ is totaly symmetric, we have that $c(u_i^su_j^t)=a$ for all $u_i^s\in U_i$, $u_j^t\in U_j$.

  Assume second that $b$ is a color in $S^1(K)$ saturated by $v_i$. We claim that $b$ is saturated by $u_i^1$ in $G^c$ as well. Since $b\in S^1(K)$, $b$ cannot be saturated by any vertices other than $u_i^1$ in $G^c$. If $b$ is not saturated by $u_i^1$ in $G^c$, then $b$ is not saturated by $U_i$ in $G^c$ by Lemma \ref{LeExtremal} (1). Thus there is an edge $u_{j_1}^su_{j_2}^t$ with $j_1,j_2\neq i$ and $c(u_{j_1}^su_{j_2}^t)=b$. Since $G^c$ is totaly symmetric and $b$ is not saturated by $u_{j_1}^s,u_{j_2}^t$, we see that all edges between $U_{j_1}$ and $U_{j_2}$ are of color $b$, specially $c(u_{j_1}^1u_{j_2}^1)=b$. It follows that $c(v_{j_1}v_{j_2})=b$, contradicting that $b$ is saturated by $v_i$ in $K^c$. Now as we claimed, $b$ is saturated by $u_i^1$ in $G^c$. Since $G^c$ is totaly symmetric, there are colors $b_1=b,b_2,\ldots,b_{n_i}\in C(G)$ such that $c(u_i^su_j^t)=b_s$ when ever $c(u_i^1u_j^1)=b$ (or equally, $c(v_iv_j)=b$).

  Assume third that $c$ is a color in $S^2(K)$ saturated by $v_i,v_j$. We claim that $c$ is saturated by $u_i^1,u_j^1$ in $G^c$ as well. Suppose otherwise. By Lemma \ref{LeExtremal} (1), there is a vertex $u_\ell^t\in U_\ell$ with $\ell\neq i,j$ such that $c\in C_G(u_\ell^t)$. Recall that $c(u_i^1u_j^1)=c$, implying that $c$ is not saturated by $u_\ell^t$. Since $G^c$ is totaly symmetric, we have that $c\in C_G(u_\ell^1)$, and then $c\in C_K(v_\ell)$, contradicting that $c\in S^2(K)$. Thus as we claimed, $u_i^1u_j^1$ is the unique edge of $G$ of color $c$. Since $G^c$ is totaly symmetric, there are $n_in_j$ colors $c_{s,t}$, $1\leq s\leq n_i$, $1\leq t\leq n_j$, where $c_{1,1}=c$, such that $c(u_i^su_j^t)=c_{s,t}$.

  By the analysis above, we see that $G^c$ is the blow-up of $K^c$. This proves the upper bound of the theorem.

  On the other hand, let $c$ be an arbitrary coloring of $K$ without rainbow $H=K_k$ and let $G^c=\mathcal{B}(K^c,f)$. If $G^c$ contains a rainbow $K_k$, say with vertex set $\{u_{i_1}^{s_1},u_{i_2}^{s_2},\ldots,u_{i_k}^{s_k}\}$, then $\{v_{i_1},v_{i_2},\ldots,v_{i_k}\}$ induces a rainbow $K_k$ in $K$, a contradiction. Thus we have that $G^c$ contains no rainbow $K_k$, which proves the lower bound of the theorem.
\end{proof}

\section{Proofs of Theorems \ref{ThkPartite} and \ref{ThBalanced}}

\begin{proof}[Proof of Theorem \ref{ThkPartite}]
In this theorem we deal with the case $k=r$. We call the following coloring $c$ of $G$ the \emph{normal coloring}: First let $K^c$ be a colored $K=K_k$ such that all edges have distinct colors with the only exception $c(v_{k-2}v_k)=c(v_{k-1}v_k)$, and then let $G^c=\mathcal{B}(K^c,f)$. It follows that the normal coloring $c$ of $G$ contains no rainbow $K_k$ with $c(G)=\sum_{1\leq i<j\leq k}n_in_j-n_k(n_{k-1}+n_{k-2}-1)$. This proves the lower bound of the theorem.

On the other hand, we let $c$ be an arbitrary coloring of $K$ which is not rainbow, and let $G^c=\mathcal{B}(K^c,f)$. There are two edges of the same color in $K^c$, implying that $S^0(K)\cup S^1(K)\neq\emptyset$.

If there is a color $a\in S^0(K)$, then there are at least two edges $v_{i_1}v_{j_1},v_{i_2}v_{j_2}$ of color $a$. By the definition of the blow-up of $K^c$, all edges of $G$ between $U_{i_1},U_{j_1}$, and between $U_{i_2},U_{j_2}$, have the same color $a$ in $G^c$. This implies that $$c(G)\leq\sum_{1\leq i<j\leq k}n_in_j-n_{i_1}n_{j_1}-n_{i_2}n_{j_2}+1\leq\sum_{1\leq i<j\leq k}n_in_j-n_k(n_{k-1}+n_{k-2}-1).$$
If there is a color $b\in S^1(K)$, say saturated by $v_{i_1}$, then there are two vertices $v_{j_1},v_{j_2}\in V(K)$ with $c(v_{i_1}v_{j_1})=c(v_{i_1}v_{j_2})=b$. By the definition of the blow-up of $K^c$, there are $n_{i_1}$ colors $b_1,b_2,\ldots,b_{n_{i_1}}$ such that all edges of $G^c$ between $U_{i_1}$ and $U_{j_1}\cup U_{j_2}$, have the colors in $\{b_1,b_2,\ldots,b_{n_{i_1}}\}$. This implies that $$c(G)\leq\sum_{1\leq i<j\leq k}n_in_j-n_{i_1}(n_{j_1}+n_{j_2}-1)\leq\sum_{1\leq i<j\leq k}n_in_j-n_k(n_{k-1}+n_{k-2}-1).$$

In each case we have that $c(G)\leq\sum_{1\leq i<j\leq k}n_in_j-n_k(n_{k-1}+n_{k-2}-1)$, which shows the upper bond of the theorem by Theorem \ref{ThUnbalanced}.
\end{proof}

\begin{proof}[Proof of Theorem \ref{ThBalanced}]
In this theorem we deal with the case $n_1=n_2=\cdots n_r=t$. If $k=r$, then the result can be deduced by Theorem \ref{ThkPartite} directly. So we only deal with the case $r>k$.

By Theorem \ref{ThUnbalanced}, $\ar(G,H)=\max\{c(\mathcal{B}(K^c,f)): K^c\mbox{ contains no rainbow }K_k\}$. Let $T_{n,k}$ denote the Tur\'an graph, which is the complete $k$-partite graph of order $n$ such that each two partite sets have size difference at most one. Tur\'an's theorem states that $\ex(K_n,K_{k+1})=e(T_{n,k})$. We define the \emph{Tur\'an coloring} $c$ of $G=K_r^t$ as follows: First let $c$ be a coloring of $K=K_r$ with a rainbow $T_{r,k-2}$ and the edges in $E(K)\backslash E(T_{r,k-2})$ having one extra common color, and then let $G^c=\mathcal{B}(K^c,f)$. One can compute that the Tur\'an coloring of $G$ contains no rainbow $H$ and has $t^2\ex(K_r,K_{k-1})+1$ colors. This proves the lower bond of the theorem.

On the other hand, let $c$ be an arbitrary  coloring of $K$ without rainbow $H$, and let $G^c=\mathcal{B}(K^c,f)$. We will show that $c(G)\leq t^2\tau+1$, where $\tau:=\ex(K_r,K_{k-1})$. We set $s_i=|S^i(K)|$, $i=0,1,2$. Since $K^c$ contains no rainbow $K_k$, by Theorem \ref{ThSc}, $c(K)=s_0+s_1+s_2\leq\tau+1$. Recall that $f(v_i)=t$ for all $i$, $1\leq i\leq r$. By the definition of the blow-up of $K^c$, $c(G)=s_0+ts_1+t^2s_2$.

If $t=1$, then clearly $c(G)=c(K)\leq\tau+1$. Now suppose that $t\geq 2$. If $s_0+s_1=0$, then $K$ is rainbow and contains a rainbow $H$, a contradiction. Thus we have that $s_0+s_1\geq 1$. If $s_0+s_1\geq 2$, then $s_2\leq\tau-1$ and $c(G)=s_0+ts_1+t^2s_2\leq 2t+t^2(\tau-1)<t^2\tau+1$. Now assume that $s_0+s_1=1$.

If $s_0=0$ and $s_1=1$, then the only non-exclusive color of $K^c$ is saturated by exactly one vertex, say $v_i$. It follows that $K-v_i$ is rainbow and contains a rainbow $H$ (recall that we assume that $r>k$), a contradiction. Thus we conclude that $s_0=1$ and $s_1=0$. Thus $c(G)=s_0+ts_1+t^2s_2\leq t^2\tau+1$. This shows the upper bond of the theorem.
\end{proof}

\section{The extremal colorings}

In this section we consider the extremal colorings for the anti-Ramsey numbers of the complete graph $H$ in the complete multi-partite graphs $G$.

Two colored graphs $G_1^{c_1}$ and $G_2^{c_2}$ are \emph{isomorphic} if there are bijections $\rho: V(G_1)\rightarrow V(G_2)$ and $\sigma: C_1(G_1)\rightarrow C_2(G_2)$ such that: (1) $uv\in E(G_1)$ if and only if $\rho(u)\rho(v)\in E(G_2)$; and (2) if $uv\in E(G_1)$, then $\sigma(c_1(uv))=c_2(\rho(u)\rho(v))$.

\subsection{For $K_k$ in complete $k$-partite graphs}

For the case $r=k$, we recall that the normal coloring of $G=K_{n_1,\ldots,n_k}$ is an extremal coloring for $\ar(G,H)$. However, there may have other extremal colorings.

The book $B_n=K_{n,1,1}$ is the graph consisting of $n$ triangles common to an edge. Set $V(B_n)=\{x,y,z_1,z_2,\ldots,z_n\}$ where $xyz_ix$, $1\leq i\leq n$, are $n$ triangles. By Theorem \ref{ThFaGyLiXi} or \ref{ThkPartite}, $\ar(B_n,K_3)=n+1$.

\medskip\noindent\textbf{Construction 1.} Let $c$ be a coloring of $B_n$ with $C(B_n)=\{a_0,a_1,\ldots,a_n\}$ such that (1) $c(xy)=a_0$, and (2) for every $z_i$, $1\leq i\leq n$, either $c(xz_i)=c(yz_i)=a_i$, or $c(xz_i)=a_i, c(yz_i)=a_0$, or $c(xz_i)=a_0, c(yz_i)=a_i$.

\medskip Notice that a coloring of $B_n$ contains no rainbow $K_3$ if and only if its every triangle has two edges of the same color. Thus the extremal colorings for $\ar(B_n,K_3)$ are exactly those we described in Construction 1. There are four non-isomorphic extremal colorings for $\ar(B_2,K_3)$ (see Figure 1). We notice that the coloring $c_1$ of $B_2$ in Figure 1 are normal.

\begin{center}
\begin{picture}(350,100)
\thicklines

\put(20,30){\put(0,30){\color{green}\line(1,-1){30}} \put(0,30){\color{red}\line(1,0){60}} \put(0,30){\color{blue}\line(1,1){30}} \put(30,0){\color{red}\line(1,1){30}} \put(30,60){\color{red}\line(1,-1){30}}
\put(0,30){\circle*{4}} \put(30,0){\circle*{4}} \put(30,60){\circle*{4}} \put(60,30){\circle*{4}}
\put(22,-20){$B_2^{c_1}$} }

\put(100,30){\put(0,30){\color{green}\line(1,-1){30}} \put(0,30){\color{red}\line(1,0){60}} \put(0,30){\color{red}\line(1,1){30}} \put(30,0){\color{red}\line(1,1){30}} \put(30,60){\color{blue}\line(1,-1){30}}
\put(0,30){\circle*{4}} \put(30,0){\circle*{4}} \put(30,60){\circle*{4}} \put(60,30){\circle*{4}}
\put(22,-20){$B_2^{c_2}$} }

\put(180,30){\put(0,30){\color{green}\line(1,-1){30}} \put(0,30){\color{red}\line(1,0){60}} \put(0,30){\color{blue}\line(1,1){30}} \put(30,0){\color{red}\line(1,1){30}} \put(30,60){\color{blue}\line(1,-1){30}}
\put(0,30){\circle*{4}} \put(30,0){\circle*{4}} \put(30,60){\circle*{4}} \put(60,30){\circle*{4}}
\put(22,-20){$B_2^{c_3}$} }

\put(260,30){\put(0,30){\color{green}\line(1,-1){30}} \put(0,30){\color{red}\line(1,0){60}} \put(0,30){\color{blue}\line(1,1){30}} \put(30,0){\color{green}\line(1,1){30}} \put(30,60){\color{blue}\line(1,-1){30}}
\put(0,30){\circle*{4}} \put(30,0){\circle*{4}} \put(30,60){\circle*{4}} \put(60,30){\circle*{4}}
\put(22,-20){$B_2^{c_4}$} }

\end{picture}

\small Figure 1. Extremal colorings for $\ar(B_2,K_3)$.
\end{center}

\begin{theorem}\label{ThExkpartite}
  Let $c$ be an extremal coloring for $\ar(G,H)$, where $G=K_{n_1,\ldots,n_k}$, $H=K_k$, $k\geq 3$ and $n_1\geq\cdots\geq n_k$. Then $G^c$ satisfies one of the following constructions (up to isomorphism):\\
  (1) $k\geq 4$, $n_{k-3}=1$, and each two edges have distinct colors in $G^c$ with the only exception $c(u_{k-3}^1u_{k-2}^1)=c(u_{k-1}^1u_k^1)$;\\
  (2) $n_{k-1}=1$, the last three partite sets induced $B_{n_{k-2}}^c$ as in Construction 1, and each two edges have distinct colors in $G^c$ unless they are both in $B_{n_{k-2}}^c$;\\
  (3) for every vertex $u_k^s\in U_k$, there are two partite sets $U_i,U_j$, $1\leq i<j<k$, with $n_i+n_j=n_{k-2}+n_{k-1}$ such that the edges between $u_k^s$ and $U_i\cup U_j$ have the same color; and apart from that, all edges have distinct colors in $G^c$.
\end{theorem}

We remark that if $n_{k-3}>n_{k-2}$ and $n_{k-1}\geq 2$, then the extremal coloring for $\ar(G,H)$ is unique.

\begin{proof}
  We distinguish two cases based on the value of $n_k$.

  \begin{case}
    $n_k=1$.
  \end{case}

  In this case $U_k$ has only one vertex $u_k$. We prove the case by induction on the order of $G$. If $n_1=1$, then $G=H=K_k$, and $\ar(G,H)=e(G)-1$. Thus an edge-coloring $c$ of $G$ is extremal for $\ar(G,H)$ if and only if there are two edges $e_1,e_2$ of the same color and apart from that, all edges have distinct colors in $G^c$. If $e_1,e_2$ are nonadjacent, then $G^c$ satisfies (1), and if $e_1,e_2$ are adjacent, then $G^c$ satisfies (3) (by possibly reordering the partite sets of the same size). So we assume that $n_1\geq 2$. Since $n_k=1$, there is a smallest index $\ell$, $1\leq\ell<k$, such that $n_\ell>n_{\ell+1}$.

  \begin{subcase}
    $1\leq\ell\leq k-3$.
  \end{subcase}

  Let $u_\ell\in U_\ell$ and $G'=G-u_\ell$. Thus $G'=K_{n_1,\ldots,n_\ell-1,\ldots,n_k}$. From Theorem \ref{ThkPartite}, one can compute that $\ar(G,H)-\ar(G',H)=n_\ell(\sum_{i=1}^kn_i-n_\ell)=d_G(u_\ell)=:d$. If $d^s_G(u_\ell)<d$, then $c(G')>\ar(G',H)$ and $G'^c$ contains a rainbow $H$, a contradiction. Thus we have that $d^s_G(u_\ell)=d$ and by Lemma \ref{LeExtremal} (2), $d^s_G(u'_\ell)=d$ for every $u'_\ell\in U_\ell$. This implies that every edges incident to $u'_\ell$ has an exclusive color for every $u'_\ell\in U_\ell$.

  Notice that $c(G')=\ar(G',H)$. We have that $c$ (restricting on $G'$) is an extremal coloring for $\ar(G',H)$. By induction hypothesis, $G'^c$ satisfies (1)(2) or (3). We denote by $U'_i$, $1\leq i\leq k$, the partite sets of $G'$ as describing in (1)(2)(3) (the partite sets of $G$ and $G'$ may have different order in case some sets have the same size). Since $n_\ell\geq 2$, there is a vertex $u'_\ell\in U_i$ other than $u_\ell$ such that all edges incident to $u'_\ell$ has an exclusive color in $G'^c$. It follows that $u'_\ell\notin\bigcup_{i=k-3}^kU'_i$ for (1) and $u'_\ell\notin\bigcup_{i=k-2}^kU_i$ for (2)(3). Since all edges incident to $u_\ell$ has an exclusive color in $G^c$, we see that $G^c$ satisfies (1)(2) or (3) as well.

  \begin{subcase}
    $\ell=k-2$.
  \end{subcase}

  Let $u_{k-2}\in U_{k-2}$ and $G'=G-u_{k-2}$. Since $n_1=n_2=\ldots=n_{k-2}$, by possibly reorder the partite sets, we can assume that $u_{k-2}$ has the smallest saturated color degree in $G^c$ among all vertices in $\bigcup_{i=1}^{k-2}U_i$. From Theorem \ref{ThkPartite}, one can compute that $\ar(G,H)-\ar(G',H)=n_{k-2}(\sum_{i=1}^kn_i-n_{k-2})-1=d_G(u_{k-2})-1=:d$. If $d^s_G(u_{k-2})<d$, then $c(G')>\ar(G',H)$ and $G'^c$ contains a rainbow $H$, a contradiction. Thus we have that $d^s_G(u_{k-2})\geq d$.

  Suppose now that $d^s_G(u_{k-2})=d$, then $c(G')=\ar(G',H)$, and $d^s_G(u'_{k-2})=d$ for every $u'_{k-2}\in U_{k-2}$ by Lemma \ref{LeExtremal} (2). By induction hypothesis, $G'^c$ satisfies (2) or (3) (notice that $n_{k-3}\geq 2$). We denote by $U'_i$, $1\leq i\leq k$, the partite sets of $G'$ as describing in (2)(3). Recall that $n_i=n_{k-2}$ for $1\leq i<k-2$ and $|U'_{k-2}|<n_{k-2}$. This implies that $\bigcup_{i=1}^{k-3}U_i=\bigcup_{i=1}^{k-3}U'_i$, and all the edges incident to $\bigcup_{i=1}^{k-3}U'_i$ have exclusive colors in $G'^c$.

  Since $d_G^s(u_{k-2})=d=d(u_{k-2})-1$, there is exactly one non-exclusive color $a\in C_G(u_{k-2})$, and either $a$ is saturated by $u_{k-2}$ and appears twice at $u_{k-2}$, or $a$ is not saturated by $u_{k-2}$ and appears once at $u_{k-2}$. We now claim that all the edges incident to $\bigcup_{i=1}^{k-3}U_i$ have exclusive colors in $G^c$. Suppose there is a non-exclusive color $a'\in C_G(u_j)$ with $u_j\in U_j$, $1\leq j\leq k-3$. Recall that all edges incident to $u_j$ has an exclusive color in $G'$. This implies that $a'\in C_G(u_{k-2})$ and thus $a'=a$. Let $G''=G-u_j-(U_{k-2}\backslash\{u_{k-2}\})$. Then all edges incident to $(\bigcup_{i=1}^{k-3}U_i\backslash\{u_j\})\cup\{u_{k-2}\}$ have exclusive colors in $G''$. This implies that $G''$ contains a rainbow $H$, a contradiction. Thus as we claimed, all the edges incident to $\bigcup_{i=1}^{k-3}U_i$ have exclusive colors in $G^c$. Specially every edge of color $a$ is incident to a vertex in $U_{k-1}\cup U_k$.

  Since all the edges incident to $\bigcup_{i=1}^{k-3}U_i$ have exclusive colors, we have that $G^c[\bigcup_{i=k-2}^kU_i]$ contains no rainbow $K_3$; for otherwise $G^c$ contains a rainbow $K_k$. If $n_{k-1}=1$, then $G[\bigcup_{i=k-2}^kU_i]$ is a book $B_{n_{k-2}}$ with an extremal coloring for $\ar(B_{n_{k-2}},K_3)$. Thus $G^c$ satisfies (2).

  Now assume that $n_{k-1}\geq 2$, and then $G'^c$ satisfies (3). Since $n_{k-2}>n_{k-1}\geq 2$, $U_{k-2}\backslash\{u_{k-2}\}\neq U'_k$. If $c(u_{k-2}u_k)=c(u_{k-1}u_k)$ where $u_{k-1}\in U_{k-1}$, then $c(u_{k-2}u_k)=a$ and $G^c$ satisfies (3). So assume that $c(u_{k-2}u_k)\neq c(u_{k-1}u_k)$. Since $d_G^s(u_{k-2})=d=d_G(u_{k-2})-1$, there is at least one vertex $u'_{k-1}\in U_{k-1}$ with $c(u_{k-2}u'_{k-1})\notin\{c(u_{k-2}u_k),c(u_{k-1}u_k)\}$. Thus $G[\bigcup_{i=k-2}^kU_i]$ contains a rainbow $K_3$ and $G^c$ contains a rainbow $H$, a contradiction.

  Finally suppose that $d^s(u_{k-2})=d+1=d(u_{k-2})$. Then for every vertex $u_i\in U_i$ with $1\leq i\leq k-2$, $d_G^s(u_i)=d_G(u_i)$. This implies that all the edges incident to $\bigcup_{i=1}^{k-2}U_i$ have exclusive colors in $G^c$, and $G^c$ contains a rainbow $H$, a contradiction.

  \begin{subcase}
    $\ell=k-1$.
  \end{subcase}

  Let $u_{k-1}\in U_{k-1}$ and $G'=G-u_{k-1}$. Since $n_1=n_2=\ldots=n_{k-1}$, we can assume that $u_{k-1}$ has the smallest saturated color degree in $G^c$ among all vertices in $\bigcup_{i=1}^{k-1}U_i$. From Theorem \ref{ThkPartite}, one can compute that $\ar(G,H)-\ar(G',H)=n_{k-1}(\sum_{i=1}^kn_i-n_{k-1})-1=d_G(u_{k-1})-1:=d$. If $d^s_G(u_{k-1})<d-1$, then $c(G')>\ar(G',H)$ and $G'^c$ contains a rainbow $H$, a contradiction. Thus we have that $d^s_G(u_{k-1})\geq d$.

  If $d^s_G(u_{k-1})=d$, then $c(G')=\ar(G,H)$, and $d^s_G(u'_{k-1})=d$ for every $u'_{k-1}\in U_{k-1}$ by Lemma \ref{LeExtremal} (2). By induction hypothesis, $G'^c$ satisfies (2) or (3). We claim that $G'^c$ satisfies (3). Suppose otherwise that $n_{k-2}=n_{k-1}=2$, and $G'^c$ satisfies (2) which is not normal. Let $u'_{k-1}$ be the vertex in $U_{k-1}\backslash\{u_{k-1}\}$. It follows that $d_{G'}^s(u'_{k-1})\leq d-1$, and thus $d_G^s(u'_{k-1})\leq d-1$, a contradiction. Thus we have that $G'^c$ satisfies (3). Since $d_G^s(u_{k-1})=d=d(u_{k-1})-1$, there is exactly one non-exclusive color $a\in C_G(u_{k-1})$, and either $a$ is saturated by $u_{k-1}$ and appears twice at $u_{k-1}$, or $a$ is not saturated by $u_{k-1}$ and appears once at $u_{k-1}$. If $c(u_{k-1}u_k)=c(u'_{k-1}u_k)$, then $c(u_{k-1}u_k)=a$ and $G^c$ satisfies (3). So assume that $c(u_{k-1}u_k)\neq c(u'_{k-1}u_k)$. There is at least one vertex $u_{k-2}\in U_{k-2}$ with the color not in $\{c(u_{k-1}u_k),c(u_{k-2}u_k)\}$. Thus $G^c$ contains a rainbow $H$, a contradiction.

  Now we suppose that $d^s(u_{k-1})=d+1=d(u_{k-2}^1)$. Then for every vertex $u_i\in U_i$ with $1\leq i\leq k-1$, $d_G^s(u_i)=d_G(u_i)$. This implies that all the edges incident to $\bigcup_{i=1}^{k-1}U_i$ have exclusive colors in $G^c$, and $G^c$ contains a rainbow $H$, a contradiction.

  \begin{case}
    $n_k\geq 2$.
  \end{case}

  In this case every partite set has size at least 2. By Lemma \ref{LeExtremal} every two vertices in a common partite sets have the same saturated color degree in $G^c$; and by Lemma \ref{LeSymmetrization} any symmetrization of $c$ is an extremal coloring for $\ar(G,H)$ as well. We first prove the following claim.

  \begin{claim}\label{ClNoS0}
    Every color in $G^c$ is saturated by at least one vertex of $G$.
  \end{claim}

  \begin{proof}
  Suppose that there is a color $a\in S^0(G)$, and let $c(u_{i_1}u_{j_1})=a$ for $u_{i_1}\in U_{i_1},u_{j_1}\in U_{j_1}$. Since $c$ is extremal for $\ar(G,H)$, when we recolor the edge $u_{i_1}u_{j_1}$ with an extra color, then there will be a rainbow $H$, in which there must be an edge $u_{i_2}u_{j_2}$ of color $a$ and $H$ contains all of the vertices $u_{i_1},u_{j_1},u_{i_2},u_{j_2}$.

  If $u_{i_1}u_{j_1}$ and $u_{i_2}u_{j_2}$ are adjacent, say $u_{i_1}=u_{i_2}$, then $u_{j_2}\in U_{j_2}$ with $j_2\neq j_1$. Let $c'$ be the coloring of $G$ obtained from $c$ by doing the symmetrization operation repeatedly, at every vertex in $U_{i_1}\backslash\{u_{i_1}\}$ to $u_{i_1}$. Recall that $a$ is not saturated by $u_{i_1}$. It follows that all edges between $U_{i_1}$ and $\{u_{j_1},u_{j_2}\}$ have color $a$ in $G^{c'}$. This implies that $a$ is not saturated by $u_{j_1}$ and $u_{j_2}$ in $G^{c'}$. Let $c''$ be the coloring of $G$ obtained from $c'$ by doing the symmetrization operation repeatedly, at every vertex in $U_{j_1}\backslash\{u_{j_1}\}$ to $u_{j_1}$, and at every vertex in $U_{j_2}\backslash\{u_{j_2}\}$ to $u_{j_2}$. It follows that $c''$ is an extremal coloring for $\ar(G,H)$, and all edges between $U_{i_1}$ and $U_{j_1}\cup U_{j_2}$ have color $a$ in $G^{c''}$. Thus $c''(G)\leq\sum_{1\leq i<j\leq k}n_in_j-n_{i_1}(n_{j_1}+n_{j_2})+1<\ar(G,H)$, a contradiction.

  Now suppose that $u_{i_1}u_{j_1}$ and $u_{i_2}u_{j_2}$ are nonadjacent. Let $u_{i_2}\in U_{i_2},u_{j_2}\in U_{j_2}$ where $U_{i_1},U_{j_1},U_{i_2},U_{j_2}$ are four different partite sets. Let $c'$ be the coloring of $G$ obtained from $c$ by doing the symmetrization operation repeatedly, at every vertex in $U_{i_1}\backslash\{u_{i_1}\}$ to $u_{i_1}$, at every vertex in $U_{j_1}\backslash\{u_{j_1}\}$ to $u_{j_1}$, at every vertex in $U_{i_2}\backslash\{u_{i_2}\}$ to $u_{i_2}$, and at every vertex in $U_{j_2}\backslash\{u_{j_2}\}$ to $u_{j_2}$. Then $c'$ is an extremal coloring for $\ar(G,H)$. We have that in $G^{c'}$, all edges between $U_{i_1},U_{j_1}$ and between $U_{i_2},U_{j_2}$, have the same color $a$. Thus $c'(G)\leq\sum_{1\leq i<j\leq k}n_in_j-n_{i_1}n_{j_1}-n_{i_2}n_{j_2}+1<\ar(G,H)$, a contradiction.
  \end{proof}

  \begin{claim}\label{ClTwoSets}
    If $a\in S^1(G)$ is saturated by $u_{i_1}\in U_{i_1}$, then there are two partite sets $U_{j_1},U_{j_2}$ such that all the edges of color $a$ are those between $u_{i_1}$ and $U_{j_1}\cup U_{j_2}$.
  \end{claim}

  \begin{proof}
  Since $a$ is saturated by $u_{i_1}$, all edges of color $a$ are incident to $u_{i_1}$. If all edges of color $a$ are between $a$ and a partite sets $U_{j_1}$, then $a$ is saturated by a vertex in $U_{j_1}$ by Lemma \ref{LeExtremal} (1), contradicting $a\in S^1(G)$. Suppose now that there are three edges $u_iu_{j_1},u_iu_{j_2},u_iu_{j_3}$ of color $a$, where $u_{j_1}\in U_{j_1},u_{j_2}\in U_{j_2},u_{j_3}\in U_{j_3}$ are in different partite sets. Let $c'$ be the coloring of $G$ obtained from $c$ by doing the symmetrization operation repeatedly, at every vertex in $U_{i_1}\backslash\{u_{i_1}\}$ to $u_{i_1}$, and at every vertex in $U_{j_\ell}\backslash\{u_{j_\ell}\}$ to $u_{j_\ell}$, $\ell=1,2,3$. Then $c'$ is an extremal coloring for $\ar(G,H)$. However $c'(G)\leq\sum_{1\leq i<j\leq k}n_in_j-n_i(n_{j_1}+n_{j_2}+n_{j_3}-1)<\ar(G,H)$, a contradiction. Thus we conclude that there are two partite sets $U_{j_1},U_{j_2}$ such that all edges of color $a$ are between $u_{i_1}$ and $U_{j_1}\cup U_{j_2}$.

  Let $c(u_{i_1}u_{j_1})=c(u_{i_1}u_{j_2})=a$ with $u_{j_1}\in U_{j_1},u_{j_2}\in U_{j_2}$. Suppose that there is a vertex $u'_{j_2}\in U_{j_2}$ such that $c(u_{i_1}u'_{j_2})\neq a$. Let $c'$ be the coloring of $G$ obtained from $c$ by doing the symmetrization operation repeatedly, at every vertex in $U_{j_1}\backslash\{u_{j_1}\}$ to $u_{j_1}$, and at every vertex in $U_{j_2}\backslash\{u'_{j_2}\}$ to $u'_{j_2}$. Then $c'$ is an extremal coloring for $\ar(G,H)$, all edges between $u_{i_1}$ and $U_{j_1}$ have color $a$ and all edges between $u_{i_1}$ and $U_{j_2}$ have color other than $a$. Since $n_{j_1}\geq 2$, $a$ is saturated by $U_{j_1}$ but not saturated by any vertex in $U_{j_1}$, contradicting Lemma \ref{LeExtremal} (1). Thus we conclude that all edges between $u_{i_1}$ and $U_{j_1}\cup U_{j_2}$ have color $a$.
  \end{proof}

  \begin{claim}\label{ClCommon}
    All the vertices that saturate some colors in $S^1(G)$ are in a common partite sets.
  \end{claim}

  \begin{proof}
  Suppose that $u_{i_1}\in U_{i_1}$ saturates $a_1\in S^1(G)$, $u_{i_2}\in U_{i_2}$ saturates $a_2\in S^1(G)$, where $U_{i_1}\neq U_{i_2}$. Clearly $a_1\neq a_2$. By Claim \ref{ClTwoSets}, there are partite sets $U_{j_1},U_{j'_1},U_{j_2},U_{j'_2}$, where $j_1\neq j'_1, j_2\neq j'_2$, such that all edges between $u_{i_1}$ and $U_{j_1}\cup U_{j'_1}$ have color $a_1$, and all edges between $u_{i_2}$ and $U_{j_2}\cup U_{j'_2}$ have color $a_2$.

  Either $c(u_{i_1}u_{i_2})\neq a_1$ or $c(u_{i_1}u_{i_2})\neq a_2$. This implies that either $u_{i_1}\notin U_{j_2}\cup U_{j'_2}$ or $u_{i_2}\notin U_{j_1}\cup U_{j'_1}$. Assume without loss of generality that $u_{i_2}\notin U_{j_1}\cup U_{j'_1}$, i.e., $i_2\neq j_1,j'_1$. Let $c'$ be the coloring of $G$ obtained from $c$ by doing the symmetrization operation repeatedly, at every vertex in $U_{i_1}\backslash\{u_{i_1}\}$ to $u_{i_1}$. Now $c'$ is an extremal coloring for $\ar(G,H)$, and for every vertex $u'_{i_1}\in U_{i_1}$, all edges between $u'_{i_1}$ and $U_{j_1}\cup U_{j'_1}$ have the same color. Recall that all edges between $u_{i_2}$ and $U_{j_2}\cup U_{j'_2}$ have the same color. We have that $c'(G)\leq\sum_{1\leq i<j\leq k}n_in_j-n_{i_1}(n_{j_1}+n_{j'_1}-1)-(n_{j_2}+n_{j'_2}-1)<\ar(G,H)$, a contradiction.
  \end{proof}

  By Claim \ref{ClNoS0}, $S^0(G)=\emptyset$. If $S^1(G)=\emptyset$, then $G$ is rainbow, a contradiction. So $S^1(G)\neq\emptyset$. By Claim \ref{ClCommon}, let $U_\ell$ be the partite set that contains all vertices saturating some colors in $S^1(G)$.

  \begin{claim}
    Every vertex in $U_\ell$ saturates exactly one color in $S^1(G)$.
  \end{claim}

  \begin{proof}
  Let $u_\ell\in U_\ell$ be arbitrary. If $u_\ell$ is not saturated by any color, then every edge incident to $u_\ell$ has an exclusive color. This implies that $d_G^s(u_\ell)=d_G(u_\ell)$, contradicting Lemma \ref{LeExtremal} (2). Suppose now that $u_\ell$ saturates two colors $a_1,a_2\in S^1(G)$. By Claim \ref{ClTwoSets}, there are partite sets $U_{j_1},U_{j'_1},U_{j_2},U_{j'_2}$ such that all edges between $u_\ell$ and $U_{j_1}\cup U_{j'_1}$ have color $a_1$ and all edges between $u_\ell$ and $U_{j_2}\cup U_{j'_2}$ have color $a_2$. Let $c'$ be the coloring of $G$ obtained from $c$ by doing the symmetrization operation repeatedly, at every vertex in $U_\ell\backslash\{u_\ell\}$ to $u_\ell$. Now $c'$ is an extremal coloring for $\ar(G,H)$, and $c'(G)\leq\sum_{1\leq i<j\leq k}n_in_j-n_\ell(n_{j_1}+n_{j'_1}+n_{j_2}+n_{j'_2}-2)<\ar(G,H)$, a contradiction.
  \end{proof}

  Now let $u_\ell\in U_\ell$ saturate $a\in S^1(G)$, and $U_{j_1},U_{j_2}$ be two partite sets such that all edges between $u_\ell$ and $U_{j_1}\cup U_{j_2}$ have color $a$. Let $c'$ be an edge-coloring of $G$ obtained from $c$ by doing the symmetrization operation repeatedly, at every vertex in $U_\ell\backslash\{u_\ell\}$ to $u_\ell$. Now $c'$ is an extremal coloring for $\ar(G,H)$, and $c'(G)=\sum_{1\leq i<j\leq k}n_in_j-n_\ell(n_{j_1}+n_{j_2}-1)=\ar(G,H)$. This implies that $n_\ell(n_{j_1}+n_{j_2}-1)=n_k(n_{k-1}+n_{k-2}-1)$. One can compute that $n_\ell=n_k$, and $n_{j_1}+n_{j_2}=n_{k-1}+n_{k-2}$.

  By possibly reordering the partite sets of the same size, we can see that $G^c$ satisfies (3). The proof is complete.
\end{proof}

\subsection{For $K_k$ in balanced $r$-partite graphs}

  Now we consider the case $G=K_r^t$. If $r=k$, then the extremal colorings for $\ar(G,H)$ are described in Theorem \ref{ThExkpartite}. So we assume that $r>k$. If $t=1$, then $G=K_r$. We recall that the Tur\'an coloring of $K_r$ (that obtained from $T_{r,k-2}$) is an extremal coloring for $\ar(K_r,K_k)$. However, it seems not easy to describe all extremal colorings for $\ar(K_r,K_k)$. In fact there are 23 extremal colorings for $\ar(K_5,K_4)$ (see Figure 2, where each black edge assigns an exclusive color). In the following we will deal with the case $t\geq 2$.

\begin{center}
\setlength{\unitlength}{0.8pt}
\begin{picture}(575,365)
\put(0,0){\line(1,0){575}} \put(0,0){\line(0,1){365}} \put(0,365){\line(1,0){575}} \put(575,0){\line(0,1){365}}

\thicklines
\put(40,280){\put(15,0){\color{red}\line(1,0){50}} \put(15,0){\line(-1,3){15}} \put(65,0){\line(1,3){15}} \put(15,0){\line(1,3){25}} \put(65,0){\line(-1,3){25}} \put(15,0){\line(13,9){65}} \put(65,0){\line(-13,9){65}} \put(0,45){\color{red}\line(1,0){80}} \put(0,45){\color{red}\line(4,3){40}} \put(80,45){\color{red}\line(-4,3){40}}
\put(15,0){\circle*{4}} \put(65,0){\circle*{4}} \put(0,45){\circle*{4}} \put(80,45){\circle*{4}} \put(40,75){\circle*{4}} }

\put(140,280){\put(15,0){\color{red}\line(1,0){50}} \put(15,0){\color{red}\line(-1,3){15}} \put(65,0){\color{red}\line(1,3){15}} \put(15,0){\line(1,3){25}} \put(65,0){\line(-1,3){25}} \put(15,0){\line(13,9){65}} \put(65,0){\line(-13,9){65}} \put(0,45){\color{red}\line(1,0){80}} \put(0,45){\line(4,3){40}} \put(80,45){\line(-4,3){40}}
\put(15,0){\circle*{4}} \put(65,0){\circle*{4}} \put(0,45){\circle*{4}} \put(80,45){\circle*{4}} \put(40,75){\circle*{4}} }

\put(240,280){\put(15,0){\line(1,0){50}} \put(15,0){\color{red}\line(-1,3){15}} \put(65,0){\color{red}\line(1,3){15}} \put(15,0){\line(1,3){25}} \put(65,0){\line(-1,3){25}} \put(15,0){\line(13,9){65}} \put(65,0){\line(-13,9){65}} \put(0,45){\line(1,0){80}} \put(0,45){\color{red}\line(4,3){40}} \put(80,45){\color{red}\line(-4,3){40}}
\put(15,0){\circle*{4}} \put(65,0){\circle*{4}} \put(0,45){\circle*{4}} \put(80,45){\circle*{4}} \put(40,75){\circle*{4}} }

\put(340,280){\put(15,0){\color{red}\line(1,0){50}} \put(15,0){\color{red}\line(-1,3){15}} \put(65,0){\color{red}\line(1,3){15}} \put(15,0){\line(1,3){25}} \put(65,0){\line(-1,3){25}} \put(15,0){\line(13,9){65}} \put(65,0){\line(-13,9){65}} \put(0,45){\line(1,0){80}} \put(0,45){\color{blue}\line(4,3){40}} \put(80,45){\color{blue}\line(-4,3){40}}
\put(15,0){\circle*{4}} \put(65,0){\circle*{4}} \put(0,45){\circle*{4}} \put(80,45){\circle*{4}} \put(40,75){\circle*{4}} }

\put(440,280){\put(15,0){\color{red}\line(1,0){50}} \put(15,0){\color{red}\line(-1,3){15}} \put(65,0){\color{red}\line(1,3){15}} \put(15,0){\line(1,3){25}} \put(65,0){\line(-1,3){25}} \put(15,0){\line(13,9){65}} \put(65,0){\line(-13,9){65}} \put(0,45){\color{blue}\line(1,0){80}} \put(0,45){\color{blue}\line(4,3){40}} \put(80,45){\line(-4,3){40}}
\put(15,0){\circle*{4}} \put(65,0){\circle*{4}} \put(0,45){\circle*{4}} \put(80,45){\circle*{4}} \put(40,75){\circle*{4}} }

\put(10,190){\put(15,0){\color{red}\line(1,0){50}} \put(15,0){\color{red}\line(-1,3){15}} \put(65,0){\line(1,3){15}} \put(15,0){\line(1,3){25}} \put(65,0){\line(-1,3){25}} \put(15,0){\color{red}\line(13,9){65}} \put(65,0){\line(-13,9){65}} \put(0,45){\line(1,0){80}} \put(0,45){\color{blue}\line(4,3){40}} \put(80,45){\color{blue}\line(-4,3){40}}
\put(15,0){\circle*{4}} \put(65,0){\circle*{4}} \put(0,45){\circle*{4}} \put(80,45){\circle*{4}} \put(40,75){\circle*{4}} }

\put(105,190){\put(15,0){\color{red}\line(1,0){50}} \put(15,0){\color{red}\line(-1,3){15}} \put(65,0){\color{blue}\line(1,3){15}} \put(15,0){\line(1,3){25}} \put(65,0){\line(-1,3){25}} \put(15,0){\color{red}\line(13,9){65}} \put(65,0){\line(-13,9){65}} \put(0,45){\color{blue}\line(1,0){80}} \put(0,45){\line(4,3){40}} \put(80,45){\line(-4,3){40}}
\put(15,0){\circle*{4}} \put(65,0){\circle*{4}} \put(0,45){\circle*{4}} \put(80,45){\circle*{4}} \put(40,75){\circle*{4}} }

\put(200,190){\put(15,0){\color{red}\line(1,0){50}} \put(15,0){\color{red}\line(-1,3){15}} \put(65,0){\color{blue}\line(1,3){15}} \put(15,0){\line(1,3){25}} \put(65,0){\line(-1,3){25}} \put(15,0){\color{red}\line(13,9){65}} \put(65,0){\line(-13,9){65}} \put(0,45){\line(1,0){80}} \put(0,45){\line(4,3){40}} \put(80,45){\color{blue}\line(-4,3){40}}
\put(15,0){\circle*{4}} \put(65,0){\circle*{4}} \put(0,45){\circle*{4}} \put(80,45){\circle*{4}} \put(40,75){\circle*{4}} }

\put(295,190){\put(15,0){\color{red}\line(1,0){50}} \put(15,0){\color{red}\line(-1,3){15}} \put(65,0){\color{blue}\line(1,3){15}} \put(15,0){\line(1,3){25}} \put(65,0){\line(-1,3){25}} \put(15,0){\color{red}\line(13,9){65}} \put(65,0){\line(-13,9){65}} \put(0,45){\line(1,0){80}} \put(0,45){\color{blue}\line(4,3){40}} \put(80,45){\line(-4,3){40}}
\put(15,0){\circle*{4}} \put(65,0){\circle*{4}} \put(0,45){\circle*{4}} \put(80,45){\circle*{4}} \put(40,75){\circle*{4}} }

\put(390,190){\put(15,0){\color{red}\line(1,0){50}} \put(15,0){\color{blue}\line(-1,3){15}} \put(65,0){\line(1,3){15}} \put(15,0){\line(1,3){25}} \put(65,0){\line(-1,3){25}} \put(15,0){\line(13,9){65}} \put(65,0){\line(-13,9){65}} \put(0,45){\color{blue}\line(1,0){80}} \put(0,45){\color{red}\line(4,3){40}} \put(80,45){\color{red}\line(-4,3){40}}
\put(15,0){\circle*{4}} \put(65,0){\circle*{4}} \put(0,45){\circle*{4}} \put(80,45){\circle*{4}} \put(40,75){\circle*{4}} }

\put(485,190){\put(15,0){\color{red}\line(1,0){50}} \put(15,0){\color{blue}\line(-1,3){15}} \put(65,0){\line(1,3){15}} \put(15,0){\line(1,3){25}} \put(65,0){\line(-1,3){25}} \put(15,0){\line(13,9){65}} \put(65,0){\color{blue}\line(-13,9){65}} \put(0,45){\line(1,0){80}} \put(0,45){\color{red}\line(4,3){40}} \put(80,45){\color{red}\line(-4,3){40}}
\put(15,0){\circle*{4}} \put(65,0){\circle*{4}} \put(0,45){\circle*{4}} \put(80,45){\circle*{4}} \put(40,75){\circle*{4}} }

\put(10,100){\put(15,0){\color{red}\line(1,0){50}} \put(15,0){\color{blue}\line(-1,3){15}} \put(65,0){\line(1,3){15}} \put(15,0){\line(1,3){25}} \put(65,0){\line(-1,3){25}} \put(15,0){\color{blue}\line(13,9){65}} \put(65,0){\line(-13,9){65}} \put(0,45){\line(1,0){80}} \put(0,45){\color{red}\line(4,3){40}} \put(80,45){\color{red}\line(-4,3){40}}
\put(15,0){\circle*{4}} \put(65,0){\circle*{4}} \put(0,45){\circle*{4}} \put(80,45){\circle*{4}} \put(40,75){\circle*{4}} }

\put(105,100){\put(15,0){\color{red}\line(1,0){50}} \put(15,0){\color{blue}\line(-1,3){15}} \put(65,0){\color{blue}\line(1,3){15}} \put(15,0){\line(1,3){25}} \put(65,0){\line(-1,3){25}} \put(15,0){\line(13,9){65}} \put(65,0){\line(-13,9){65}} \put(0,45){\line(1,0){80}} \put(0,45){\color{red}\line(4,3){40}} \put(80,45){\color{red}\line(-4,3){40}}
\put(15,0){\circle*{4}} \put(65,0){\circle*{4}} \put(0,45){\circle*{4}} \put(80,45){\circle*{4}} \put(40,75){\circle*{4}} }

\put(200,100){\put(15,0){\color{green}\line(1,0){50}} \put(15,0){\line(-1,3){15}} \put(65,0){\line(1,3){15}} \put(15,0){\color{blue}\line(1,3){25}} \put(65,0){\color{blue}\line(-1,3){25}} \put(15,0){\line(13,9){65}} \put(65,0){\line(-13,9){65}} \put(0,45){\color{green}\line(1,0){80}} \put(0,45){\color{red}\line(4,3){40}} \put(80,45){\color{red}\line(-4,3){40}}
\put(15,0){\circle*{4}} \put(65,0){\circle*{4}} \put(0,45){\circle*{4}} \put(80,45){\circle*{4}} \put(40,75){\circle*{4}} }

\put(295,100){\put(15,0){\line(1,0){50}} \put(15,0){\color{green}\line(-1,3){15}} \put(65,0){\color{green}\line(1,3){15}} \put(15,0){\color{blue}\line(1,3){25}} \put(65,0){\color{blue}\line(-1,3){25}} \put(15,0){\line(13,9){65}} \put(65,0){\line(-13,9){65}} \put(0,45){\line(1,0){80}} \put(0,45){\color{red}\line(4,3){40}} \put(80,45){\color{red}\line(-4,3){40}}
\put(15,0){\circle*{4}} \put(65,0){\circle*{4}} \put(0,45){\circle*{4}} \put(80,45){\circle*{4}} \put(40,75){\circle*{4}} }

\put(390,100){\put(15,0){\line(1,0){50}} \put(15,0){\color{green}\line(-1,3){15}} \put(65,0){\line(1,3){15}} \put(15,0){\color{blue}\line(1,3){25}} \put(65,0){\color{blue}\line(-1,3){25}} \put(15,0){\color{green}\line(13,9){65}} \put(65,0){\line(-13,9){65}} \put(0,45){\line(1,0){80}} \put(0,45){\color{red}\line(4,3){40}} \put(80,45){\color{red}\line(-4,3){40}}
\put(15,0){\circle*{4}} \put(65,0){\circle*{4}} \put(0,45){\circle*{4}} \put(80,45){\circle*{4}} \put(40,75){\circle*{4}} }

\put(485,100){\put(15,0){\color{green}\line(1,0){50}} \put(15,0){\color{green}\line(-1,3){15}} \put(65,0){\line(1,3){15}} \put(15,0){\color{blue}\line(1,3){25}} \put(65,0){\color{blue}\line(-1,3){25}} \put(15,0){\line(13,9){65}} \put(65,0){\line(-13,9){65}} \put(0,45){\line(1,0){80}} \put(0,45){\color{red}\line(4,3){40}} \put(80,45){\color{red}\line(-4,3){40}}
\put(15,0){\circle*{4}} \put(65,0){\circle*{4}} \put(0,45){\circle*{4}} \put(80,45){\circle*{4}} \put(40,75){\circle*{4}} }

\put(10,10){\put(15,0){\color{blue}\line(1,0){50}} \put(15,0){\color{blue}\line(-1,3){15}} \put(65,0){\color{green}\line(1,3){15}} \put(15,0){\color{green}\line(1,3){25}} \put(65,0){\line(-1,3){25}} \put(15,0){\line(13,9){65}} \put(65,0){\line(-13,9){65}} \put(0,45){\line(1,0){80}} \put(0,45){\color{red}\line(4,3){40}} \put(80,45){\color{red}\line(-4,3){40}}
\put(15,0){\circle*{4}} \put(65,0){\circle*{4}} \put(0,45){\circle*{4}} \put(80,45){\circle*{4}} \put(40,75){\circle*{4}} }

\put(105,10){\put(15,0){\color{blue}\line(1,0){50}} \put(15,0){\color{blue}\line(-1,3){15}} \put(65,0){\line(1,3){15}} \put(15,0){\line(1,3){25}} \put(65,0){\color{green}\line(-1,3){25}} \put(15,0){\color{green}\line(13,9){65}} \put(65,0){\line(-13,9){65}} \put(0,45){\line(1,0){80}} \put(0,45){\color{red}\line(4,3){40}} \put(80,45){\color{red}\line(-4,3){40}}
\put(15,0){\circle*{4}} \put(65,0){\circle*{4}} \put(0,45){\circle*{4}} \put(80,45){\circle*{4}} \put(40,75){\circle*{4}} }

\put(200,10){\put(15,0){\color{blue}\line(1,0){50}} \put(15,0){\color{blue}\line(-1,3){15}} \put(65,0){\color{green}\line(1,3){15}} \put(15,0){\line(1,3){25}} \put(65,0){\line(-1,3){25}} \put(15,0){\color{green}\line(13,9){65}} \put(65,0){\line(-13,9){65}} \put(0,45){\line(1,0){80}} \put(0,45){\color{red}\line(4,3){40}} \put(80,45){\color{red}\line(-4,3){40}}
\put(15,0){\circle*{4}} \put(65,0){\circle*{4}} \put(0,45){\circle*{4}} \put(80,45){\circle*{4}} \put(40,75){\circle*{4}} }

\put(295,10){\put(15,0){\line(1,0){50}} \put(15,0){\color{green}\line(-1,3){15}} \put(65,0){\color{blue}\line(1,3){15}} \put(15,0){\line(1,3){25}} \put(65,0){\color{green}\line(-1,3){25}} \put(15,0){\color{blue}\line(13,9){65}} \put(65,0){\line(-13,9){65}} \put(0,45){\line(1,0){80}} \put(0,45){\color{red}\line(4,3){40}} \put(80,45){\color{red}\line(-4,3){40}}
\put(15,0){\circle*{4}} \put(65,0){\circle*{4}} \put(0,45){\circle*{4}} \put(80,45){\circle*{4}} \put(40,75){\circle*{4}} }

\put(390,10){\put(15,0){\color{green}\line(1,0){50}} \put(15,0){\line(-1,3){15}} \put(65,0){\color{blue}\line(1,3){15}} \put(15,0){\color{green}\line(1,3){25}} \put(65,0){\line(-1,3){25}} \put(15,0){\color{blue}\line(13,9){65}} \put(65,0){\line(-13,9){65}} \put(0,45){\line(1,0){80}} \put(0,45){\color{red}\line(4,3){40}} \put(80,45){\color{red}\line(-4,3){40}}
\put(15,0){\circle*{4}} \put(65,0){\circle*{4}} \put(0,45){\circle*{4}} \put(80,45){\circle*{4}} \put(40,75){\circle*{4}} }

\put(485,10){\put(15,0){\line(1,0){50}} \put(15,0){\color{green}\line(-1,3){15}} \put(65,0){\color{blue}\line(1,3){15}} \put(15,0){\line(1,3){25}} \put(65,0){\line(-1,3){25}} \put(15,0){\color{blue}\line(13,9){65}} \put(65,0){\color{green}\line(-13,9){65}} \put(0,45){\line(1,0){80}} \put(0,45){\color{red}\line(4,3){40}} \put(80,45){\color{red}\line(-4,3){40}}
\put(15,0){\circle*{4}} \put(65,0){\circle*{4}} \put(0,45){\circle*{4}} \put(80,45){\circle*{4}} \put(40,75){\circle*{4}} }

\end{picture}

\small Figure 2. Extremal colorings for $\ar(K_5,K_4)$.
\end{center}

We denote by $K_k-e$ the graph obtained from $K_k$ by removing an arbitrary edge. The \emph{hourglass}, the \emph{house}, and the \emph{prism} are graphs shown in Figure 3.

\begin{center}
\begin{picture}(220,60)

\thicklines
\put(10,10){\put(0,0){\line(0,1){40}} \put(0,0){\line(5,4){50}} \put(0,40){\line(5,-4){50}} \put(50,0){\line(0,1){40}}
\put(0,0){\circle*{4}} \put(0,40){\circle*{4}} \put(25,20){\circle*{4}} \put(50,0){\circle*{4}} \put(50,40){\circle*{4}} }

\put(80,10){\put(0,0){\line(0,1){40}} \put(0,0){\line(1,0){25}} \put(0,40){\line(1,0){25}} \put(25,0){\line(0,1){40}} \put(25,0){\line(5,4){25}} \put(25,40){\line(5,-4){25}}
\put(0,0){\circle*{4}} \put(0,40){\circle*{4}} \put(25,0){\circle*{4}} \put(25,40){\circle*{4}} \put(50,20){\circle*{4}} }

\put(150,10){\put(0,0){\line(0,1){40}} \put(0,0){\line(1,1){20}} \put(0,40){\line(1,-1){20}} \put(0,0){\line(1,0){60}} \put(0,40){\line(1,0){60}} \put(20,20){\line(1,0){20}} \put(40,20){\line(1,-1){20}} \put(40,20){\line(1,1){20}} \put(60,0){\line(0,1){40}}
\put(0,0){\circle*{4}} \put(0,40){\circle*{4}} \put(20,20){\circle*{4}} \put(40,20){\circle*{4}} \put(60,0){\circle*{4}} \put(60,40){\circle*{4}} }
\end{picture}

\small Figure 3. The hourglass, the house and the prism.
\end{center}

We make use of the following result.

\begin{theorem}[Dirac \cite{Di}]\label{ThDi}
  If $k\geq 4$ and $n\geq k+1$, then
  $$\ex(K_n,K_k-e)=\ex(K_n,K_{k-1})=e(T_{n,k-2}),$$
  and the extremal graph for $\ex(K_n,K_k-e)$ is either the Tur\'an graph $T_{n,k-2}$, or the hourglass or house (for $k=4,n=5$), or the prism (for $k=4,n=6$).
\end{theorem}

\begin{theorem}
  Let $c$ be an extremal coloring for $\ar(G,H)$, where $G=K_r^t$, $H=K_k$, $r>k\geq 4$, $t\geq 2$. If $c$ is totaly symmetric, then $G^c=\mathcal{B}(K^c,f)$, where $K=K_r$ and $f(v)=t$ for all $v\in V(K)$, such that one of the following holds (up to isomorphism):\\
  (1) $k=4,r=5$, and $c$ is a coloring of $K$ with a rainbow hourglass or house and all other edges having an extra common color;\\
  (2) $k=4,r=6$, and $c$ is a coloring of $K$ with a rainbow prism and all other edges having an extra common color;\\
  (3) $c$ is a coloring of $K$ with a rainbow Tur\'an graph $T_{r,k-2}$ and all other edges having an extra common color.
\end{theorem}

\begin{proof}
  From the analysis in the proof of Theorem \ref{ThUnbalanced}, we see that $G^c$ is the blow-up of a colored graph $K^c$, where $K^c$ contains no rainbow $H$. Set $s^i(K)=|S^i(K)|$, $i=0,1,2$. By Theorem \ref{ThSc}, $c(K)=s_0+s_1+s_2\leq\ex(K_r,K_{k-1})$; and by the definition of the blow-up, $c(G)=s_0+ts_1+t^2s_2=t^2\ex(K_r,K_{k-1})+1$. From the analysis in proof of Theorem \ref{ThBalanced}, we see that $s_0=1$, $s_1=0$ and $s_2=\ex(K_r,K_{k-1})$.

  Let $L$ be the subgraph of $K$ induced by all the edges of colors in $S^2(K)$. Thus $e(L)=s_2=\ex(K_r,K_{k-1})$. If $L$ contains a $K_k-e$, then the unique missing edge $e$ has the color of $S^0$ in $K^c$, and thus $K$ contains a rainbow $H$, a contradiction. This implies that $L$ contains no $K_k-e$. By Theorem \ref{ThDi}, $L$ is either the Tur\'an graph $T_{r,k-2}$, or the hourglass or house (for $k=4,r=5$), or the prism (for $k=4,r=6$). That is, $K^c$ satisfies (1)(2) or (3).
\end{proof}

We remark that there are extremal colorings for $\ar(G,H)$ that are not totaly symmetric.

\medskip\noindent\textbf{Example 1.} Suppose that $k-2\nmid r$. Then the Tur\'an graph $T_{r-1,k-2}$ has at least two partite sets of size $\lfloor\frac{r-1}{k-2}\rfloor$. Let $T=T_{r-1,k-2}$ with $V(T)=\{v_1,\ldots,v_{r-1}\}$, $T_1,T_2$ be two graphs isomorphic to $T_{r,k-2}$ obtained from $T$ by adding a vertex $v_r^1,v_r^2$, respectively, such that $N_{T_1}(v_r^1)\neq N_{T_2}(v_r^2)$. Let $t=t_1+t_2$, $U_i=\{u_i^1,\ldots,u_i^t\}$, $1\leq i\leq r-1$, $U_r^1=\{u_r^1,\ldots,u_r^{t_1}\}$, $U_r^2=\{u_r^{t_1+1},\ldots,u_r^t\}$. Let $c$ be a coloring of $G=K_r^t$ such that all edges in $$\{u_i^su_j^{s'}: v_iv_j\in E(T)\}\cup\{u_i^su_r^{s'}: s'\leq t_1,v_iv_r^1\in E(T_1)\mbox{ or }s'>t_1,v_iv_r^2\in E(T_2)\}$$
have exclusive colors and all other edges have an extra common color. Then $c$ is an extremal coloring for $\ar(G,H)$ which is not totaly symmetric.

\section{Acknowledgements}

The research of the first and third authors was supported by NSFC (12071370) and Shaanxi Fundamental Science Research Project for Mathematics
and Physics (22JSZ009). The research of the second author was supported by NKFIH, grant K132696. This work was done when the third author visited the
Alfr\'{e}d R\'{e}nyi Institute of Mathematics.

\end{document}